\documentclass[notitlepage]{scrartcl}
\usepackage[shortlabels]{enumitem}
\usepackage{amsmath}
\usepackage{amssymb}
\usepackage{amsthm}
\usepackage{abstract}
\usepackage{xcolor}
\usepackage{comment}
\usepackage{mathtools}
\usepackage{graphicx}
\usepackage{caption}
\usepackage{subcaption}
\usepackage{float}
\usepackage{hyperref}
\hypersetup{
    colorlinks=true,
    linkcolor=blue,
    filecolor=magenta,      
    urlcolor=cyan
    }
\makeatletter
\newcommand*{\barfix}[2][.175ex]{%
  \mathpalette{\@barfix{#1}}{#2}%
}
\newcommand*{\@barfix}[3]{%
  \vbox{%
    \kern#1\relax
    \hbox{$#2#3\m@th$}%
  }%
}
\makeatother

\newtheorem{theorem}{Theorem}
\newtheorem{thm}{Theorem}[section]

\newtheorem{lemma}[thm]{Lemma}
\newtheorem{claim}[thm]{Claim}
\newtheorem{remark}[thm]{Remark}

\newtheorem{question}[thm]{Question}
\newcommand{\footremember}[2]{%
    \footnote{#2}
    \newcounter{#1}
    \setcounter{#1}{\value{footnote}}%
}

\usepackage[margin=1in]{geometry}

\title{\vspace{-1.5cm}Percolation through Isoperimetry} 
\author{%
Sahar Diskin \footremember{alley}{\scriptsize{School of Mathematical Sciences, Tel Aviv University, Tel Aviv 6997801, Israel. Email: sahardiskin@mail.tau.ac.il.}}%
\and Joshua Erde \footremember{trailer}{\scriptsize{Institute of Discrete Mathematics, Graz University of Technology, Steyrergasse 30, 8010 Graz, Austria. Email: erde@math.tugraz.at. Research supported in part by FWF P36131. For the purpose of open access, the author has applied a CC BY public copyright licence to any Author Accepted Manuscript version arising from this submission.}}%
\and Mihyun Kang \footremember{alley2}{\scriptsize{Institute of Discrete Mathematics, Graz University of Technology, Steyrergasse 30, 8010 Graz, Austria. Email: kang@math.tugraz.at. Research supported in part by FWF W1230 and I6502.}}%
\and Michael Krivelevich \footremember{trailer2}{\scriptsize{School of Mathematical Sciences, Tel Aviv University, Tel Aviv 6997801, Israel. Email: krivelev@tauex.tau.ac.il. Research supported in part by USA–Israel BSF grant 2018267.}}%
}
\begin{document}
\maketitle

\begin{abstract}
We provide a sufficient condition on the isoperimetric properties of a regular graph $G$ of growing degree $d$, under which the random subgraph $G_p$ typically undergoes a phase transition around $p=\frac{1}{d}$ which resembles the emergence of a giant component in the binomial random graph model $G(n,p)$. We further show that this condition is tight. 

More precisely, let $d=\omega(1)$, let $\epsilon>0$ be a small enough constant, and let $p \cdot d=1+\epsilon$. We show that if $C$ is sufficiently large and $G$ is a $d$-regular $n$-vertex graph where every subset $S\subseteq V(G)$ of order at most $\frac{n}{2}$ has edge-boundary of size at least $C|S|$, then $G_p$ typically has a unique linear sized component, whose order is asymptotically $y(\epsilon)n$, where $y(\epsilon)$ is the survival probability of a Galton-Watson tree with offspring distribution Po$(1+\epsilon)$. We further give examples to show that this result is tight both in terms of its dependence on $C$, and with respect to the order of the second-largest component.

We also consider a more general setting, where we only control the expansion of sets up to size $k$. In this case, we show that if $G$ is such that every subset $S\subseteq V(G)$ of order at most $k$ has edge-boundary of size at least $d|S|$ and $p$ is such that $p\cdot d \geq 1 + \epsilon$, then $G_p$ typically contains a component of order $\Omega(k)$.
\end{abstract}

\section{Introduction}

\subsection{Background and Motivation}
In 1957, Broadbent and Hammersley \cite{BH57} initiated the study of percolation theory in order to model the flow of fluid through a medium with randomly blocked channels. In \emph{(bond) percolation}, given a \textit{host graph} $G$, the \textit{percolated random subgraph} $G_p$ is obtained by retaining every edge of $G$ independently and with probability $p$. A fundamental and extensively studied example of such a model is when the host graph $G$ is the complete graph $K_n$, which coincides with the classical \textit{binomial random graph model} $G(n,p)$.

In their seminal paper from 1960, Erd\H{o}s and R\'enyi \cite{ER60} showed that $G(n,p)$\footnote{In fact, Erd\H{o}s and R\'enyi worked in the closely related \emph{uniform} random graph model $G(n,m)$.} undergoes a dramatic phase transition, with respect to its component structure, when the expected average degree is around $1$ (that is, when $p\cdot n\approx 1$).  More precisely, given a constant $\epsilon>0$, let us define $y\coloneqq y(\epsilon)$ to be the unique solution in $(0,1)$ of the equation
\begin{align}\label{survival prob}
    y=1-\exp\left\{-(1+\epsilon)y\right\}.
\end{align}
Erd\H{o}s and R\'enyi showed the following phase transition result.
\begin{thm}[\cite{ER60}]\label{ER thm}
Let $\epsilon>0$ be a small enough constant. Then, with probability tending to $1$ as $n$ tends to infinity,
\begin{enumerate}[(a)]
    \item\label{i:sub} if $p=\frac{1-\epsilon}{n}$, then all components of $G(n,p)$ are of order $O_{\epsilon}\left(\log n \right)$; and,
    \item\label{i:sup} if $p=\frac{1+\epsilon}{n}$, then there exists a unique \emph{giant} component in $G(n,p)$ of order $(1+o(1))y(\epsilon)n$. Furthermore, all other components of $G(n,p)$ are of order $O_{\epsilon}\left(\log n\right)$.
\end{enumerate}
\end{thm}
We refer the reader to \cite{B01, FK16, JLR00} for systematic coverage of random graphs, and to the monographs \cite{BR06, G99, K82} for more background on percolation theory. 

Let us say a little about a proof of Theorem \ref{ER thm}. A standard heuristic (see, for example, \cite[Chapter 10.4]{AS16}) shows that, since every vertex in the host graph has degree $\approx n$, a local exploration process of the clusters in $G(n,p)$ can be approximated by a branching process where the expected number of children is $pn$. From this, one can deduce that each vertex is contained in a `large' component, in this case linear sized, with probability $\approx y(\epsilon)$ where $\epsilon=np-1$, and what remains is to show that these large clusters `merge' into a unique `giant' component (throughout the rest of the paper, we will often refer to a linear sized component as a giant component).

This heuristic implies that the threshold probability depends only on the local structure of the graph, the vertex degrees, 
whereas the order of the largest component above this threshold grows with the order of the host graph. In the case of $G(n,p)$, these two parameters are coincidentally approximately equal. However, more generally this heuristic suggests that we should see a sharp change in the order of the largest component in a random subgraph of an arbitrary $d$-regular host graph $G$, for $d=\omega(1)$, when $p$ is around $\frac{1}{d}$.

In the subcritical regime, it is known (and it is fairly easy to show) that for $d$-regular graphs the behaviour of Theorem \ref{ER thm} \ref{i:sub} is universal, in that when $p=\frac{1-\epsilon}{d}$, typically all components of $G_p$ are of order at most $O_{\epsilon}(\log |V(G)|)$ (see \cite[Theorem 1]{DEKK22} and also \cite{NP10}). In the supercritical regime, consider first a disjoint union of copies of $K_{d+1}$ --- the largest component (whether we percolate or not) is then of order at most $d$. Thus, for such a phase transition to typically occur, one needs some additional conditions on the edge distribution of the host graph. Indeed, this phase transition has been studied in certain specific families of host graphs. In particular, in the case of the $d$-dimensional hypercube $Q^d$, the pioneering work of Ajtai, Koml\'os, and Szemer\'edi~\cite{AKS81} and of Bollob\'as, Kohayakawa, and \L{}uczak~\cite{BKL92} shows that the component structure of a supercritical percolated hypercube is quantitatively similar to that in Theorem \ref{ER thm} \ref{i:sup}, in that the largest component has asymptotic order $y(\epsilon)|V(Q^d)|$ and the second-largest component has order $O_\epsilon(\log |V(Q^d)|)$. This result has been further generalised by the authors to any high-dimensional product graph, whose base graphs are regular and of bounded degree \cite{DEKK22}. In a different context, this phenomenon has also been observed in pseudo-random graphs, which lack any clear geometric structure but have fairly good expansion. In this setting Frieze, Krivelevich, and Martin \cite{FKM04} showed that when $\frac{\lambda}{d}=o(1)$, supercritical percolated $(n,d,\lambda)$-graphs demonstrate similar behaviour in terms of their component structure.

In all these cases, the proofs rely heavily on the \textit{expansion properties} of the host graph. However, in this regard, high-dimensional product graphs and pseudo-random graphs are quite dissimilar. The former are well-structured with respect to their geometry and display a high level of symmetry, yet they are (generally) not very good expanders, whereas the latter are much better expanders and have very uniform edge-distribution, but lack any obvious geometric structure or symmetry.

In light of this, it is perhaps natural to ask if a sufficiently strong assumption on the expansion of the host graph $G$ alone is sufficient to guarantee a similar phase transition as in $G(n,p)$, without any further assumptions of geometric structure or pseudo-randomness, and if so, what are the minimal such requirements.
\begin{question}\label{q: refined}
    Let $d=\omega(1)$, and let $G$ be a $d$-regular graph on $n$ vertices. What requirements on the expansion of $G$ do guarantee the typical existence of a component of asymptotic size $y(\epsilon)n$ in the supercritical regime? 
\end{question}
To make this question more precise, we will need some way to quantify the expansion properties of the host graph. A well-known notion, capturing in part the expansion properties, is the \textit{isoperimetric constant} of the graph, also known as the \textit{Cheeger} constant due to its connection to the Cheeger isoperimetric constant of a Riemannian manifold \cite{C70}.

Given a graph $G=(V,E)$ and a subset $S\subseteq V(G)$, we write $\partial_G(S)$ for the size of the edge-boundary of $S$, that is, the number of edges in $G$ with one endpoint in $S$ and the other endpoint in $S^C\coloneqq V\setminus S$. The \emph{isoperimetric constant} of $G$ is defined as 
\begin{align*}
    i(G)\coloneqq \min_{S\subseteq V, 1\le |S|\le |V|/2}\left\{\frac{\partial_G(S)}{|S|}\right\}.
\end{align*}
For example, it follows from Harper's Theorem \cite{H64} that $i(Q^d)=1$, whereas if $G$ is an $(n,d,\lambda)$-graph with $\frac{\lambda}{d}=o(1)$, then standard results imply $i(G)=(1+o_d(1))\frac{d}{2}$ (see, for example, \cite{KS06}). 

Graphs with \emph{constant} degree and $i(G)>0$ are known as \emph{expander graphs} and have turned out to be very important in diverse areas of discrete mathematics and computer science (we refer the reader to \cite{HLW06} for a comprehensive survey on expander graphs and their applications). Question \ref{q: refined} has been studied in the context of constant degree expander graphs. For constant degree expanders with high-girth, the phase transition occurs around $p=\frac{1}{d-1}$ (note that when $d=\omega(1)$ and outside the critical window, the scaling of $\frac{1}{d}$ suffices, whereas in the constant degree case, this difference matters). Indeed, Alon, Benjamini, and Stacey \cite{ABS04} showed that if $G$ is an expander, then there exists at most one linear sized component in $G_p$ for any $p$ and under the further assumption that the graph has high-girth, they showed the existence of a unique giant component in the supercritical regime (that is, when $p\cdot (d-1)>1$), whose asymptotic order was determined by Krivelevich, Lubetzky, and Sudakov \cite{KLS20}. Subsequent work by Alimohammadi, Borgs, and Saberi \cite{ABS23} generalised this to digraph expanders with constant \textit{average} degree. Let us further mention the work of Benjamini, Nachmias, and Peres \cite{BNP11}, who showed that under extra assumptions on the host graph there is an $o(1)$-sharp threshold for the appearance of a giant component, and conjectured its universality. Subsequent work by Benjamini, Boucheron, Lugosi, and Rossignol \cite{BBLR12} showed the existence of a sharp threshold for the existence of a giant component of order $cn$ for any $c\in(0,1)$. 

Questions of a similar flavour to Question \ref{q: refined} have also been studied under the assumption that the host graph is transitive. In particular, Easo \cite{E22} showed necessary and sufficient conditions for the existence of a percolation threshold in vertex-transitive graphs, and Easo and Hutchcroft \cite{EH21} showed sufficient conditions for the existence of a unique giant component above the percolation threshold in vertex-transitive graphs. It is further worth noting the work of Alimohammadi, Borgs, and Saberi \cite{ABS22}, which studied the emergence of a giant component when the host graph is an expander with bounded average degree, in a more general percolation setting.

\subsection{Main results}
Towards answering Question \ref{q: refined}, our first main result shows that a fairly weak assumption on the isoperimetric constant of the host graph is sufficient to guarantee the existence of a component whose asymptotic order is arbitrarily close to $y(\epsilon)n$. 
\begin{theorem}\label{th: weak expander}
Let $d=\omega(1)$, and let $G$ be a $d$-regular $n$-vertex graph. Let $\epsilon>0$ be a small constant and let $p=\frac{1+\epsilon}{d}$. Then, there exists a constant $C_0\coloneqq C_0(\epsilon)>0$, such that for every $C\ge C_0$ the following holds. If $i(G)\ge C$, then with probability at least $1-\exp\left\{-\frac{\sqrt{C}n}{d}\right\}$,
\begin{align*}
    \big||L_1|-yn\big|\le \frac{10n}{C^{1/20}},
\end{align*}
where $y=y(\epsilon)$ is as defined in \eqref{survival prob} and $L_1$ is the largest component of $G_p$. Furthermore, the second-largest component of $G_p$ is of order at most $(1+o_d(1))\frac{10n}{C^{1/20}}$.
\end{theorem}

We note that in fact an even weaker assumption on the isoperimetry of the host graph, restricted to sets of linear size, suffices (see Remark \ref{r: key theorem}). It is worth noting here that in a recent work of Borgs and Zhao \cite{BZ23}, it was shown that under somewhat similar conditions on the isoperimetry of the host graph, together with \textit{bounded} average degree assumption, there exists a local sampling algorithm which provides, among other things, a robust estimate for the order of the largest component.  

Furthermore, it suffices to take $C_0(\epsilon)=\Omega\left(\epsilon^{-40}\right)$. Observe that for constant $C$ this does not quite answer Question \ref{q: refined}, but when taking ${C=C(d) \to \infty}$ so that $\frac{\sqrt{C}n}{d}\to\infty$, the result implies that \textbf{whp}\footnote{With high probability, that is, with probability tending to $1$ as $d$ tends to infinity.} $|L_1|=\left(1+o_d(1)\right)yn$ and all other components have sublinear order. In particular, when $G$ is an $(n,d,\lambda)$-graph with $\lambda=o(d)$, we have that $i(G)\ge \left(1-o(1)\right)\frac{d}{2}$, and so Theorem \ref{th: weak expander} implies the typical existence of a giant component of order $(1+o_d(1))yn$ in the supercritical regime, which was first shown in \cite{FKM04}. 

On the other hand, in the case of the hypercube $Q^d$, a subcube $S$ of size $2^{d-1}$ has an edge-boundary of size only $|S|$. Nevertheless, we note that an ad-hoc adjustment of our method, utilising Harper's inequality~\cite{H64}, can give a proof for the typical existence of a giant of the correct asymptotic order in supercritical $Q^d_p$ (see Remark \ref{r: key theorem}). 

Since, in general, our methods will always lead to some inverse polynomial dependence on $\epsilon$, we have not attempted to optimise the power of $C$ in the estimate of Theorem \ref{th: weak expander}, nor the power $\epsilon$ in the estimate of $C_0$. Nevertheless, the above example suggests it would be very interesting to see if the dependence here could be significantly improved. 

We remark that our proof technique of Theorem \ref{th: weak expander} carries on to $p=\frac{c}{d}$ for any constant $c>1$, showing that \textbf{whp}, $|L_1|=(1+o(1))y(c-1)n$ (under an appropriate assumption on $C_0=C_0(c)$). Furthermore, our argument allows also for $\epsilon=\epsilon(d)$ which tends slowly to zero as $d$ tends to infinity.

Theorem \ref{th: weak expander} is tight in two senses. Firstly, there exist a $d$-regular graph $G$ whose isoperimetric constant is bounded from below and some supercritical probability $p$ such that \textbf{whp} $G_p$ contains no linear sized component in the supercritical regime, and in fact only (poly-)logarithmically sized components.
\begin{theorem}\label{th: construction}
Let $C>1$ be a constant, let $d = \omega(1)$, let $n=\omega(d^2)$, and let $p\le\frac{C}{d}$. Then, there exists a $d$-regular $n$-vertex graph $G$ with $i(G) \ge \frac{1}{40C}$ such that \textbf{whp} all the components in $G_p$ are of order at most $3d\log n$.
\end{theorem}
Secondly, the bound on the size of the second-largest component in Theorem \ref{th: weak expander} is close to tight, in that there exist $d$-regular graphs $G$ whose isoperimetric constant is arbitrarily large and supercritical probabilities $p$ for which the second-largest component in $G_p$ has almost linear size.
\begin{theorem}\label{th: construction 2}
Let $d=\omega(1)$, let $C\ge 1$ with $7C\le d$, and let $n\ge 10C\cdot d$. let $\epsilon>0$ be a sufficiently small constant, and let $p=\frac{1+\epsilon}{d}$. Then, there exists a $d$-regular $n$-vertex graph $G$ with $i(G)\ge C$, such that the second-largest component of $G_p$ is of order at least $\epsilon d$ with probability at least $1-\exp\left\{-\frac{n}{\exp\{30C\}d}\right\}-o_d(1)$.
\end{theorem}

We note that in Theorem \ref{th: construction 2}, $C$ can, but does not have to, depend on $d$ and tend slowly to infinity. Furthermore, note that for any $d=o\left(\frac{n}{\exp\{30C\}}\right)$, we have that \textbf{whp} the second-largest component is of order at least $\epsilon d$. Finally, we remark that in the constructions given for Theorems \ref{th: construction} and \ref{th: construction 2}, there are some suitable parity assumptions on $n$, $d$, and $C$, which we leave implicit here in order to simplify the statements of the theorems.

As mentioned before, the connection between the phase transition in a percolated subgraph and the isoperimetric properties of the host graph has been studied before for \textit{constant}-degree high-girth expanders. To be more precise, given fixed constants $d\in \mathbb{N}$ and $i>0$, in their pioneering paper, Alon, Benjamini, and Stacey \cite{ABS04} showed that for a $d$-regular graph $G$ with $i(G) \geq i$  and high-girth, $G_p$ undergoes a phase transition around $p=\frac{1}{d-1}$, where the order of the largest component grows from sublinear to linear. Furthermore, they showed that the second-largest component in $G_p$ in the supercritical regime is typically of order $O(n^c)$, for some $c<1$.
Subsequent work by Krivelevich, Lubetzky, and Sudakov \cite{KLS20} determined that the typical asymptotic order of the largest component in the supercritical regime, when $p=\frac{1+\epsilon}{d-1}$, is $y(\epsilon)n$, and demonstrated the existence of a high-girth constant-degree expander $H$ such that the second-largest component of $H_p$ in the supercritical regime is of order $\Omega(n^c)$, for any $c<1$.

Let us briefly compare this to our results, where instead the degree $d$ is tending to infinity. Whilst we have a similar phase transition in terms of the size of the largest component when $p$ is around $\frac{1}{d}$, we require a stronger assumption on the isoperimetric constant, $i(G)\ge C$, for some large $C$. Theorem \ref{th: construction} shows that this is in some sense necessary, demonstrating a difference between the constant degree and growing degree setting. Moreover, whilst our assumption on the isoperimetric constant is stronger, let us note that any $(n,d,\lambda)$-graph $G$ satisfies $i(G)=\Theta(d)$, and that for every $S\subseteq V(Q^d)$ with $|S|\le |V(Q^d)|^{1-c}$, for any $c<1$, we have $\partial_{Q^d}(S)=\Theta(d)|S|$. It is thus perhaps surprising that Theorem \ref{th: weak expander} holds \textbf{whp} for any large enough $C$, with no dependence on $d$ in how \textit{quickly} we need $C$ to tend to infinity.

Furthermore, outside of the largest component, the typical component structure in the supercritical regime is quantitatively different --- in the constant degree setting the second-largest component is \textbf{whp} of order $O(n^c)$ for some $c<1$, while when the degree is growing, Theorem~\ref{th: construction 2} shows that the second-largest component can be typically of order $\Omega(n/t)$, for any function $t$ tending to infinity arbitrarily slowly, demonstrating a stark difference between the two settings. 

Finally, we will also consider what we can say about Question \ref{q: refined} when we have rather less control over the expansion of the host graph. Indeed, let us suppose that we do not have good control on the isoperimetric constant of $G$, or even over the expansion of linear sized sets as in Remark \ref{r: key theorem}, but rather a more `restricted' control over the expansion of $G$, up to sets of a fixed size $k$. In the spirit of Question \ref{q: refined}, can we still hope to determine the existence of some `large' component? In this case, a graph consisting of many disjoint expanding graphs of size $2k$ shows that we cannot hope to find a component of order larger than $2k$. Our second main result demonstrates that we can guarantee \textbf{whp} a component of order $\Omega(k)$, when the probability $p$ is sufficiently large with respect to our local control on the expansion of $G$.

\begin{theorem}\label{th: general statement}
Let $k=\omega(1)$, $d\le k$ and let $G$ be a graph on more than $k$ vertices, such that every $S\subseteq V(G)$ with $|S|\le k$ satisfies $\partial_G(S)\ge d|S|$. Let $\epsilon>0$ be a small constant and let $p=\frac{1+\epsilon}{d}$. Then, with probability tending to $1$ as $k$ tends to infinity, $G_p$ contains a component of order at least $\frac{k}{2}$. \label{i: up to k}
\end{theorem}

At first sight, the assumptions of Theorem \ref{th: general statement} seem quite dissimilar to other statements about the existence of large clusters in percolated subgraphs. However, for sets $S$ of size one, the condition is equivalent to the condition on the minimum degree of the host graph, $\delta(G) \geq d$, and for such graphs it is already known (see, e.g., \cite{KS13}) that when $p \geq \frac{1+\epsilon}{d}$ \textbf{whp} $G_p$ contains a component of order $\Omega(d)$. In particular, the main interest here is when $k \gg d$. In some sense, we can think of Theorem \ref{th: general statement} as giving an alternative heuristic for the nature of the phase transition --- here our point of criticality is controlled by the expansion ratio of subsets, whereas the quantitative aspects of the component structure above the critical point are controlled by the \emph{scale} on which this level of expansion holds. Indeed, in the complete graph, sets of order $k$ have edge-expansion of $\left(1-\frac{k}{n}\right)n\cdot k$. Thus, choosing $k=\delta n$, Theorem \ref{th: general statement} shows that when $p\cdot n>\frac{1}{1-\delta}$, \textbf{whp} $G(n,p)$ contains a linear sized component. Choosing $\delta$ sufficiently small, we have that whenever $p \cdot n >1$, \textbf{whp} $G(n,p)$ contains a linear sized component.

In this way, Theorem \ref{th: general statement} once again demonstrates the intrinsic connection between the expansion of the host graph, both global and `local', and the typical emergence of large components in the percolated subgraphs. Moreover, we prove a variant of Theorem \ref{th: general statement} which requires a slightly stronger assumption on the probability and on the degree of the graph, but allows one to verify instead only a `local' expansion property in the host graph (see Theorem \ref{th: general statement b}, stated in Section \ref{s: general iso}).

The paper is structured as follows. In Section \ref{s: prel} we set our notation, collect several lemmas which we will use throughout the proofs, and describe the Breadth First Search (BFS) algorithm which we will also use. In Section \ref{s: general iso} we give short proofs of Theorems \ref{th: general statement} and \ref{th: general statement b}, utilising the BFS algorithm. The proof of Theorem \ref{th: weak expander}, given in Section \ref{s: main result}, is the most involved part of the paper, and therein lie several novel ideas. In Section \ref{s: construction} we give two constructions proving Theorems \ref{th: construction} and \ref{th: construction 2}. Finally, in Section \ref{s: discussion}, we discuss our results and avenues for future research.

\section{Preliminaries}\label{s: prel}
As mentioned in the introduction, given a graph $H$ and a subset $S\subseteq V(H)$, we denote by $\partial_H(S)$ the number of edges with one endpoint in $S$ and the other endpoint in $S^C=V(H)\setminus S$. We denote by $N_H(S)$ the external neighbourhood of $S$ in $H$. Given two disjoint subsets $A,B\subseteq V(H)$, we denote by $e_H(A,B)$ the number of edges (in $H$) with one endpoint in $A$ and the other endpoint in $B$. We denote by $e_H(A)$ the number of edges in $H[A]$. Given $v\in V(H)$, we denote by $C_H(v)$ the connected component in $H$ to which $v$ belongs. In each case, if the underlying graph is clear from the context, we may omit the subscript. Furthermore, throughout the paper, we omit rounding signs for the sake of clarity of presentation.

If $y(\epsilon)$ is defined as in \eqref{survival prob}, then we note that $y(\epsilon)$ is the survival probability of a Galton-Watson tree with offspring distribution Po$(1+\epsilon)$, and it can be shown that $y(\epsilon)$ is an increasing continuous function on $(0,\infty)$ with $y(\epsilon) = 2\epsilon - O(\epsilon^2)$. 

Given an $n$-vertex graph $H$ and a subset $I\subseteq [n]$, the $I$\textit{-restricted} isoperimetric constant of $H$ is given by
\begin{align*}
    i_I(H)\coloneqq \min_{S\subseteq V(H), |S|\in I}\left\{\frac{\partial_H(S)}{|S|}\right\}.
\end{align*}
When $I=\{k\}$, we may abbreviate $i_{\{k\}}(H)$ into $i_k(H)$.\footnote{We note that in some papers $i_k$ is used to refer to $i_{[k]}$} For example, in Theorems \ref{th: weak expander} and \ref{th: general statement}, our assumptions correspond to a lower bound on the restricted isoperimetric constants $i_{[\frac{n}{2}]}(G)$ and $i_{[k]}(G)$, respectively.

We will make use of two standard probabilistic bounds. The first one is a typical Chernoff-type tail bound on the binomial distribution (see, for example, Appendix A in \cite{AS16}).
\begin{lemma}\label{l:  chernoff}
Let $n\in \mathbb{N}$, let $p\in [0,1]$, and let $X\sim Bin(n,p)$. Then for any $0<t\le \frac{np}{2}$, 
\begin{align*}
    &\mathbb{P}\left[|X-np|\ge t\right]\le 2\exp\left\{-\frac{t^2}{3np}\right\}.
\end{align*}
\end{lemma}

The second one is a variant of the well-known Azuma-Hoeffding inequality (see, for example, Chapter 7 in \cite{AS16}),
\begin{lemma}\label{l: azuma}
Let $m\in \mathbb{N}$ and let $p\in [0,1]$. Let $X = (X_1,X_2,\ldots, X_m)$ be a random vector with range $\Lambda = \{0,1\}^m$ with $X_{\ell}$ distributed according to Bernoulli$(p)$. Let $f:\Lambda\to\mathbb{R}$ be such that there exists $C \in \mathbb{R}$ such that for every $x,x' \in \Lambda$ which differ only in one coordinate,
\begin{align*}
    |f(x)-f(x')|\le C.
\end{align*}
Then, for every $t\ge 0$,
\begin{align*}
    \mathbb{P}\left[\big|f(X)-\mathbb{E}\left[f(X)\right]\big|\ge t\right]\le 2\exp\left\{-\frac{t^2}{2mpC^2}\right\}.
\end{align*}
\end{lemma}

We also require the following bound on the number of $k$-vertex trees in a graph
$G$, which follows immediately from \cite[Lemma 2]{BFM98}.
\begin{lemma}\label{l: trees}
Let $k \in \mathbb{N}$ and let $t_k(G)$ be the number of trees on $k$ vertices which are subgraphs of an $n$-vertex graph $G$. Let $d\coloneqq \Delta(G)$ be the maximum degree of $G$. Then
\begin{align*}
     t_k(G)\le n\frac{k^{k-2}d^{k-1}}{k!}\le n(ed)^{k-1}.
\end{align*}
\end{lemma}

Finally, we will utilise the following lemma, allowing one to find large matchings in percolated subgraphs, which follows immediately from \cite[Lemma 3.8]{DEKK24}.
\begin{lemma}\label{l: matching}
Let $G$ be a $d$-regular graph. Let $c_1>0$ and $0<\delta<\frac{1}{2}$ be constants. Let $s\ge c_1d$. Let $F\subseteq E(G)$ be such that $|F|\ge s$, and let $q=\frac{\delta}{d}$. Then, there exists a constant $c_2=c_2(\delta)\ge\delta^2$ such that $F_{q}$, a random subset of $F$ obtained by retaining each edge independently with probability $q$, contains a matching of size at least $\frac{c_2s}{d}$ with probability at least $1-\exp\left\{-\frac{c_2s}{d}\right\}$.
\end{lemma}

\subsection{The Breadth First Search algorithm}\label{s: BFS}
The Breadth First Search (BFS) algorithm is an algorithm which explores the components of a graph $G$ by building a maximal spanning forest. 

The algorithm receives as input a graph $G=(V,E)$ and an order $\sigma$ on $V$. The algorithm maintains three sets of vertices: 
\begin{itemize}
\item $W,$ the set of vertices whose exploration is complete; 
\item $Q,$ the set of vertices currently being explored, kept in a queue; and
\item $U,$ the set of vertices that have not been explored yet.
\end{itemize}
The algorithm starts with $W=Q=\emptyset$ and $U=V(G)$, and ends when $Q\cup U=\emptyset$. At each step, if $Q$ is non-empty, then the algorithm queries $U$ for neighbours in $G$ of the first vertex $v$ in $Q$, according to $\sigma$. Each neighbour which is discovered is added to the back of the queue $Q$. Once all neighbours of $v$ have been discovered, we move $v$ from $Q$ to $W$. If $Q=\emptyset$, then we move the next vertex from $U$ (according to $\sigma$) into $Q$. Note that the set of edges discovered during the algorithm forms a maximal spanning forest of $G$. In particular, we begin exploring a component $K$ of $G$ when the first vertex of $K$ (according to $\sigma$) enters $Q$, which was empty at that moment, and we complete exploring the component the first moment $Q$ becomes empty again.

In order to analyse the BFS algorithm on a random subgraph $G_p$ of a graph $G$ with $|E(G)|$ edges, we will utilise the \emph{principle of deferred decisions}. That is, we will take a sequence $(X_j \colon 1 \leq j \leq |E(G)|)$ of i.i.d. Bernoulli$(p)$ random variables, which we will think of as representing a positive or negative answer to a query in the algorithm. When the $j$-th edge of $G$ is queried during the BFS algorithm, we will include it in $G_p$ if and only if $X_j=1$. Note that the forest obtained in this way has the same distribution as a forest obtained by running the BFS algorithm on $G_p$. We say the $j$-th edge of $G$ is queried at \textit{time} $j$ in the process, and write $W(j), Q(j)$ and $U(j)$ for the sets $W, Q$ and $U$ directly after this edge has been queried.

\section{Large components --- assuming restricted isoperimetry}\label{s: general iso}
We begin with the proof of Theorem \ref{th: general statement}. The proof is short, drawing inspiration from \cite{KS13} and showcases how one can utilise the properties of the BFS algorithm in the setting of percolation.
\begin{proof}[Proof of Theorem \ref{th: general statement}] 
We run the BFS algorithm described in Section \ref{s: BFS} on $G_p$. Let $t_0$ be the first time that $|W(t_0)|=\frac{k}{2}$, and let $t_1$ be the first time when $|W(t_1)|=k$.

Suppose towards a contradiction that $Q$ is empty at some time $t\in[t_0, t_1]$, and let $k'\in\left[\frac{k}{2},k\right]$ be such that $|W(t)|=k'$. By our assumption on the isoperimetric properties of $G$ we have $\partial(W(t))\ge d\cdot k'$, and since each edge between $W(t)$ and $W(t)^C$ has been queried at time $t$, we have that $t\ge k'd$. 

On the other hand, at any time $t$, we have that $\sum_{j=1}^tX_j\le |W(t)\cup Q(t)|$, and in particular if $Q(t)$ is empty, then $|W(t)|\ge \sum_{j=1}^t X_j\ge \sum_{j=1}^{k'd}X_j$.
By Lemma \ref{l: chernoff},
\begin{align*}
    \mathbb{P}\left[\sum_{j=1}^{k'd}X_j\le \left(1+\frac{\epsilon}{2}\right)k'\right]\le \exp\left\{-\frac{\epsilon^2{k'}^2}{4\cdot 3(1+\epsilon)k'}\right\}\le \exp\left\{-\frac{\epsilon^2k}{25}\right\}.
\end{align*}
Recalling that $k'\in \left[\frac{k}{2},k\right]$, by the union bound over the at most $kd$ possible values of $k'd$, we have that the probability that $\sum_{j=1}^{t}X_j\le \left(1+\frac{\epsilon}{2}\right)k'$ is at most
\begin{align*}
    d k \cdot \exp\left\{-\frac{\epsilon^2k}{25}\right\}\le k^2\cdot \exp\left\{-\frac{\epsilon^2k}{25}\right\}=o(1),
\end{align*}
where the inequality follows from our assumption that $k\ge d$, and the equality follows from our assumption that $k=\omega(1)$. Thus, \textbf{whp} at any time $t\in[t_0, t_1]$, we have that $|W(t)|\ge \sum_{j=1}^tX_j\ge\sum_{j=1}^{k'd}X_j> k'$ --- a contradiction.

Therefore, \textbf{whp} $Q(t)$ is not empty for any $t\in [t_0, t_1]$, and all the vertices in $W(t_1) \setminus W(t_0)$ belong to the same component. We thus conclude that \textbf{whp} there exists a component of size at least $\frac{k}{2}$ in $G_p$.
\end{proof}

We note that in the above argument, we can choose $t_0$ to be the first moment where $|W(t_0)|=\delta k$ for small $\delta$, and with a similar analysis deduce the existence of a component of order at least $(1-o(1))k$. We have chosen $\delta = \frac{1}{2}$ for clarity of presentation.

As mentioned in the introduction, under an additional assumption on the maximum degree of $G$ and with a slightly larger probability, we can make a similar conclusion as to the size of the largest component under a much weaker isoperimetric assumption, where we only bound the expansion of sets of size \emph{exactly} $k$.
\begin{thm}\label{th: general statement b}
Let $k=\omega(1)$, let $c_1, c_2\in (0,1]$, and let $d\le c_1(1-c_2)k$. Let $G$ be a graph on more than $k$ vertices, with maximum degree $d$, and with $i_k(G)\ge c_1d$. Let $\epsilon>0$ be sufficiently small, and let $p=\frac{1+\epsilon}{c_2\cdot c_1d}$. Then, \textbf{whp}, $G_p$ contains a component of order at least $\frac{c_1(1-c_2)k}{4}$.
\end{thm}

Before proving Theorem \ref{th: general statement b}, we first prove the following lemma, which allows us to translate a bound on $i_{k_1}(G)$ to one on $i_{[k_2]}(G)$, with $k_2$ being not much smaller than $k_1$. The proof draws on ideas from \cite{K19}. 

\begin{lemma}\label{l: local to global expansion}
Let $k,d$ be positive integers, let $c_1, c_2\in [0,1]$ be such that $c_1d<k$ and let us write $c_3 := \frac{c_1(1-c_2)}{2}$. Let $G$ be a graph on more than $k$ vertices, with maximum degree $d$, and such that $i_k(G)\ge c_1d$. Then there exists a subgraph $G' \subseteq G$ with $|V(G')| \geq |V(G)|-k$, such that $i_{[c_3 k]}(G')\ge c_1c_2d$.
\end{lemma}
\begin{proof}
Initialise $G_0=G$ and $W_0=\emptyset$. At each iteration $j$, if there is a subset $B\subseteq V(G_{j-1})$ of size $|B|\le c_3 k$ with $\partial_{G_j}\left(B\right)<c_2c_1d|B|$, we update $W_j\coloneqq W_{j-1}\cup B$ and $G_j\coloneqq G_{j-1}[V(G_{j-1})\setminus B]$. We terminate this process once there are no more such subsets $B$, and let $G'$ be the resulting graph.

Suppose towards contradiction that there is some $m$ such that $|W_m|\ge k$, where we may assume without loss of generality that $m$ is minimal with this property. Let $B_0$ be the last subset added to $W_m$, so that $W_m = W_{m-1} \cup B_0$. Then, there is some $B_1\subseteq B_0$ such that $|W_{m-1} \cup B_1|=k$. Let us further set $B_2\coloneqq B_0\setminus B_1$. 

On the one hand, by our assumption on the isoperimetric inequality on $G$, that is, $i_k(G)\ge c_1d$, we have that \begin{equation}\label{e:lowerboundedges}
\partial_G\left(W_{m-1}\cup B_1\right)\ge c_1dk.
\end{equation}
On the other hand,
\begin{align*}
\partial_G\left(W_{m-1}\cup B_1\right) &\le \partial_G(W_{m-1})+e_G(B_1, V(G)\setminus\left(W_{m-1}\cup B_1)\right)\\
    &\le \partial_G(W_{m-1})+e_G\left(B_0, V(G)\setminus (W_{m-1}\cup B_0)\right)+e_G(B_1, B_2).
\end{align*}
Now, by construction, we have that $\partial_G(W_{m-1})< c_1 c_2 d |W_{m-1}| < c_2c_1dk$. Moreover, by our choice of $W_{m-1}$ and $B_0$, we have $e_G\left(B_0,V(G)\setminus (W_{m-1}\cup B_0)\right)=\partial_{G_{m-1}}(B_0)<c_1c_2d|B_0|$. Finally, since the graph $G$ has maximum degree $d$, we have that $e_G(B_1,B_2)<d|B_0|$. Recalling that $|B_0|\le c_3 k = \frac{c_1(1-c_2)k}{2}$ and $c_1 \in [0,1]$, we have that
\begin{align}
    \partial_G\left(W_{m-1}\cup B_1\right)&<c_1c_2dk+c_1c_2d\cdot \frac{c_1(1-c_2)k}{2}+d\cdot \frac{c_1(1-c_2)k}{2} \nonumber\\
    &\le c_1dk\left(c_2+\frac{(1+c_2)(1-c_2)}{2}\right)=c_1dk\left(c_2+\frac{1-c_2^2}{2}\right),\label{e:upperboundedges}
\end{align}
where in the last inequality we used our assumption that $c_1\le 1$. Observe that $f(c_2)=c_2+\frac{1-c_2^2}{2}$ is continuous, increasing on $(-\infty, 1]$, and attains the value of $1$ when $c_2=1$. Recalling that $c_2\in [0,1]$, it follows from \eqref{e:lowerboundedges} and \eqref{e:upperboundedges} that
\begin{align*}
    c_1dk\le \partial_G\left(W_{m-1}\cup B_1\right)<c_1dk\left(c_2+\frac{1-c_2^2}{2}\right)\le c_1dk,
\end{align*}
a contradiction.
\end{proof}

We can now show that Theorem \ref{th: general statement b} follows from Lemma \ref{l: local to global expansion} and Theorem \ref{th: general statement}.
\begin{proof}[Proof of Theorem \ref{th: general statement b}]
By Lemma \ref{l: local to global expansion}, there exists a non-empty subgraph $G'\subseteq G$ such that every $S\subseteq V(G')$ with $|S|\le \frac{c_1(1-c_2)k}{2}$ has $\partial_{G'}(S)\ge c_1c_2d|S|$. Applying Theorem \ref{th: general statement} to $G'$ with $p=\frac{1+\epsilon}{c_1c_2d}$, we conclude that \textbf{whp} $G'_p$ (and thus $G_p$) contains a component of size at least $\frac{c_1(1-c_2)k}{4}$.
\end{proof}

It would be interesting to see whether the maximum degree condition of Lemma \ref{l: local to global expansion}, and thus of Theorem \ref{th: general statement b}, can be replaced with an assumption on the \textit{average} degree of $G$.

\begin{remark}\label{r:long paths}
We note that one cannot hope to obtain a path of length $\Omega_{\epsilon}(k)$ under the assumptions of Theorem \ref{th: general statement}. For example, let $G=K_{d,d^{10}}$, and let $k=d^2$. We have that $i_{[k]}(G)\ge d$. All paths in $G$ have length at most $2d$, and thus, naturally, we cannot hope to find any path of length $\Omega_{\epsilon}(k)$ in a random subgraph of $G$.
\end{remark}

\section{Existence of a giant component under weak isoperimetric assumptions}\label{s: main result}
Let us start this section by defining some useful notation, and giving a broad outline of the proof of Theorem \ref{th: weak expander}. 

Our proof will proceed using a three-round exposure. Let $\delta\coloneqq \delta(\epsilon)>0$ be a sufficiently small constant. We let $p_2=p_3=\frac{\delta}{d}$, and let $p_1$ be such that $(1-p_1)(1-p_2)(1-p_3)=1-p$, so that $G_p$ has the same distribution as $G_{p_1}\cup G_{p_2}\cup G_{p_3}$. In order to easily describe the intermediary stages of the three-round exposure, we let $p'$ be such that $(1-p_1)(1-p_2)=(1-p')$, and define the following three graphs. 
\begin{itemize}
    \item $G(1)=G_{p_1}$, noting that $p_1\ge \frac{1+\epsilon-2\delta}{d}$;
    \item $G(2)=G(1)\cup G_{p_2}$, noting that $G(2) \sim G_{p'}$ and $p'\ge \frac{1+\epsilon-\delta}{d}$; and 
    \item  $G(3)=G(2)\cup G_{p_3}$, noting that $G(3) \sim G_p$.
\end{itemize}
Let us write $\rho_1=p_1, \rho_2=p'$ and $\rho_3=p$, so that $G(j) \sim G_{\rho_j}$ for each $j\in\{1,2,3\}$. Note that $V(G(j))=V(G)$ for each $j\in\{1,2,3\}$.
    
We now define several sets which will be crucial to our analysis in the section. Firstly, $V_{L}(G(j))$ is the set of vertices which lie in `large' components in $G(j)$. To be precise, 
\begin{align*}
    V_L(G(j))\coloneqq\left\{v\in V(G)\colon |C_{G(j)}(v)|\ge \delta d\right\}.
\end{align*}
Secondly, $W_{L}(G(j))$ is the set of vertices with many neighbours in $V_L(G(j))$. That is,
\begin{align*}
    W_{L}(G(j))\coloneqq \left\{w\in V(G)\colon |N_G(w)\cap V_{L}(G(j))|\ge \delta^5d\right\}.
\end{align*}
Finally, $V_{S}(G(j))$ is the set of vertices in `small' components in $G(j)$. That is,
\begin{align*}
    V_S(G(j))\coloneqq \left\{v\in V(G)\colon |C_{G(j)}(v)|\le \log^2d\right\}.
\end{align*}

Our strategy for proving Theorem \ref{th: weak expander} is then broadly as follows. The first step, which is relatively standard, is to show that \textbf{whp} the right asymptotic proportion of vertices lie in `large' components in each $G(j)$. To do this, we estimate the number of vertices which lies in `small' components and then show that a negligible proportion of the vertices lies in components of intermediate size.

The key part of the proof, which contains several novel arguments, is to show that \textbf{whp} almost all vertices have many neighbours which lie in `large' components in $G(2)$. Such a statement already appears in the seminal work of Ajtai, Koml\'os and Szemer\'edi \cite{AKS81}, and is a key part of the analysis of the phase transition in various geometric graphs \cite{AKS81,DEKK22}. However, the proofs in these settings rely heavily on the \textit{self-symmetry} of these graphs, and a key improvement here is to prove such a statement without any structural assumptions on the host graph, in fact relying only on its \textit{regularity}. This follows from a delicate argument analysing a modified BFS process starting at the neighbours of a fixed vertex. The rough idea here is that, before we have discovered $\Omega(d)$ neighbours which lie in large components, \textbf{whp} $o(d)$ vertices are contained in components which are not 'large', which will allow us to couple the exploration process from below with a supercritical branching process, and conclude that each neighbour has a constant probability of lying in a large component.

With these two results in hand, we can argue for the typical existence of a giant component, roughly as in \cite{AKS81,BKL92,DEKK22}. More precisely, we show that almost every vertex in $V_L(G(2))$ will coalesce into a single component in $G(3)$ as follows. We say that a partition $V_L(G(2))=A\cup B$ is component-respecting if $K\cap V_L(G(2))$ is fully contained in either $A$ or $B$, for every component $K$ of $G(2)$. Given a component-respecting partition $V_L(G(2))=A\cup B$, since almost every vertex is adjacent to many vertices in large components in $G(2)$, we can extend this partition to an almost partition $A' \cup B'$ of $V(G)$ such that $A \subseteq A'$, $B \subseteq B'$ and every vertex in $A', B'$ has many neighbours in $A, B$, respectively. Our assumption on the isoperimetric properties of $G$ ensures that there are many edges between $A'$ and $B'$, which we can extend to a large family of paths in $G$ between $A$ and $B$ of length at most $3$. Using Lemma \ref{l: matching}, which provides with very high probability a large matching between $A'$ and $B'$ in $G_{p_3}$, we can argue that with very high probability one of these paths is present in $G(3)$, and in fact that the failure probability is small enough that \textbf{whp} this holds for \emph{every} component-respecting partition of $V_L(G(2))$ of relevant sizes. Hence, \textbf{whp} $G(3)$ contains a component containing almost all the vertices in $V_L(G(2))$, which contains asymptotically the required number of vertices.

Let us begin then by estimating the order of $V_S(G(j))$ for each $j\in\{1,2,3\}$. Note that the number of vertices in small components is a decreasing property, and it thus suffices to upper bound the order of $V_S(G(1))$ and to lower bound the order of $V_S(G(3))$. To that end, let
\begin{align*}
    F(c)=\sum_{k=1}^{\infty}\frac{k^{k-1}}{k!}c^{k-1}\exp\left\{-ck\right\}.
\end{align*}
It is known (see for example \cite[p.~346]{ER60}) that given $c>1$, we have that 
\begin{equation}
    F(c)=1-y(c-1), \label{eq: ER crazy equality}
\end{equation}
where $y(x)$ is defined as in \eqref{survival prob}. With this in hand, we are ready to estimate $|V_S(G(j))|$, in a manner similar to the one used already in the seminal work of Erd\H{o}s and R\'enyi \cite{ER60}.
\begin{lemma}\label{l: |V_S(G(1))|}
We have that \textbf{whp}
\begin{align*}
    &|V_S(G(1))|\le \left(1+o_d(1)\right)\left(1-y(\epsilon-2\delta)\right)n,\\
    &|V_S(G(3))|\ge \left(1-o_d(1)\right)\left(1-y(\epsilon)\right)n.
\end{align*}
\end{lemma}
\begin{proof}
We begin by bounding $\mathbb{E}\left[|V_S(G(1))|\right]$ from above. Fix $k\le \log^2d$. Let $X_{k}$ be the number of vertices contained in components of order $k$ in $G(1)=G_{p_1}$. Let $\mathcal{T}_k$ denote the set of trees on $k$ vertices in $G$. Then, recalling $\rho_1=p_1$, we have
\begin{align*}
     \mathbb{E}\left[|V_S(G(1))|\right]&\le \mathbb{E}\left[\sum_{k=1}^{\log^2d}X_k\right]\le \sum_{k=1}^{\log^2d}k\sum_{T\in \mathcal{T}_k}\rho_1^{k-1}(1-\rho_1)^{e_G(T,T^C)}\\
    &\le n\sum_{k=1}^{\log^2d}\frac{(dk)^{k-1}}{k!}\rho_1^{k-1}(1-\rho_1)^{ki_k(G)},
\end{align*}
where the third inequality follows from the first inequality in Lemma \ref{l: trees}, and from the definition of the restricted isoperimetric constant $i_k(G)$. Since $k\le \log^2d$ and $G$ is $d$-regular, it follows that $i_k(G)\ge d-\log^2d$. Therefore, since $\rho_1 \geq \frac{1+\epsilon - 2\delta}{d}$ and $x\exp\left\{-ax\right\}$ is decreasing when $ax>1$, we have
\begin{align*}
    \mathbb{E}[|V_S(G(1))|]&\le n \sum_{k=1}^{\log^2d}\frac{k^{k-1}}{k!}(d\rho_1)^{k-1}\exp\left\{-\rho_1k(d-\log^2d)\right\}\\
    &\le n \sum_{k=1}^{\log^2d}\frac{k^{k-1}}{k!}(1+\epsilon-2\delta)^{k-1}\exp\left\{-(1+\epsilon-2\delta)k\left(1-\frac{\log^2d}{d}\right)\right\}\\
    &\le\exp\left\{\frac{2\log^4d}{d}\right\}n\sum_{k=1}^{\log^2d}\frac{k^{k-1}}{k!}(1+\epsilon-2\delta)^{k-1}\exp\left\{-(1+\epsilon-2\delta)k\right\} \\
    &\le (1+o_d(1))n\cdot F(1+\epsilon-2\delta)\\
    &= (1+o_d(1))\left(1-y(\epsilon-2\delta)\right)n,
\end{align*}
where the penultimate inequality and last equality follow from \eqref{eq: ER crazy equality}.

Let us now bound $|V_S(G(3))|$ from below. Since $G$ is $d$-regular, for every $v\in V(G)$ a standard coupling argument implies that $|C_{G(3)}(v)|$ is stochastically dominated by the number of vertices in a Galton-Watson tree with offspring distribution $Bin(d,\rho_3)$. Hence, standard results (see, for example, \cite[Theorem 4.3.12]{D19}) imply that for every $v\in V(G)$,
\begin{align*}
    \mathbb{P}\left[|C_{G(3)}(v)|\ge \log d\right]\le y(\epsilon)+o_d(1).
\end{align*}
Therefore,
\begin{align*}
    \mathbb{E}[|V_S(G(3))|]\ge \left(1-o_d(1)\right)\left(1-y(\epsilon)\right)n.
\end{align*}

Finally, we show that $|V_S(G(j))|$ is tightly concentrated around its mean for each $j\in \{1,2,3\}$. Indeed, let us consider the standard edge-exposure martingale on $G(j)$. Changing any edge can change the value of $|V_S(G(j))|$ by at most $2\log^2d$. Therefore, by Lemma \ref{l: azuma},
\begin{align*}
    \mathbb{P}\left[\big||V_S(G(j))|-\mathbb{E}[|V_S(G(j))|]\big|\ge n^{2/3}\right]&\le 2\exp\left\{-\frac{n^{4/3}}{2\cdot\frac{dn}{2}\cdot\rho_j\cdot(2\log^2d)^2}\right\}\\
    &\le 2\exp\left\{-\frac{n^{1/3}}{5\log^4 n}\right\}=o(1),
\end{align*}
where we used $d\le n$ in the last inequality.
\end{proof}

We now show that \textbf{whp} there are $o_d(n)$ vertices in $V(G)\setminus\left(V_S(G(j))\cup V_L(G(j))\right)$.
\begin{lemma}\label{l: gap statement}
For all $j\in \{1,2,3\}$, \textbf{whp} the number of vertices in components of order between $\log^2d$ and $\delta d$ in $G(j)$ is $o_d(n)$.
\end{lemma}
\begin{proof}
Fix $\log^2d \le k\le \delta d$. As in the proof of Lemma \ref{l: |V_S(G(1))|}, let $X^{(j)}_k$ be the number of vertices in components of order $k$ in $G(j)$. Then, by the second inequality in Lemma \ref{l: trees}, we obtain
\begin{align*}
    \mathbb{E}[X^{(j)}_k]&\le k\sum_{T\in \mathcal{T}_k}\rho_j^{k-1}(1-\rho_j)^{e_G(T,T^C)}\\
    &\le nk(ed)^{k-1}\rho_j^{k-1}(1-\rho_j)^{ki_k(G)}.
\end{align*}
Since $k\le \delta d$ and $G$ is $d$-regular, we have that $i_k(G)\ge (1-\delta)d$. Furthermore, $\rho_j \geq \frac{1+\epsilon - 2\delta}{d}$ for each $j$, and since $x\exp\left\{-ax\right\}$ is decreasing when $ax>1$,
\begin{align*}
    \mathbb{E}[X^{(j)}_k]&\le nk(ed)^{k-1}\rho_j^{k-1}(1-\rho_j)^{(1-\delta)dk}\\
    &\le nk\left[ed\rho_j\exp\left\{-(1-\delta)d\rho_j\right\}\right]^k\\
    &\le nd\left[e(1+\epsilon-2\delta)\exp\left\{-(1-\delta)(1+\epsilon-2\delta)\right\}\right]^k\\
    &\le nd\left[\left(1+\epsilon-2\delta\right)\exp\left\{-\epsilon+4\delta\right\}\right]^k.
\end{align*}
Furthermore, we have that $1+x\le \exp\left\{x-\frac{x^2}{3}\right\}$ for small enough $x$ and therefore
\begin{align*}
    \mathbb{E}[X^{(j)}_k]&\le nd\left[\exp\left\{\epsilon-2\delta-\frac{\left(\epsilon-2\delta\right)^2}{3}\right\}\exp\left\{-\epsilon+4\delta\right\}\right]^k\\
    &\le nd\exp\left\{-\left(\frac{\epsilon^2}{3}-3\delta\right)k\right\}\\
    &\le nd\exp\left\{-\epsilon^3\log^2d\right\},
\end{align*}
where we used that $k\ge \log^2d$ and that $\delta$ is small enough with respect to $\epsilon$. By Markov's inequality, we have
\begin{align*}
    \mathbb{P}\left[X^{(j)}_k\ge nd^3\exp\left\{-\epsilon^3\log^2d\right\}\right]\le\frac{1}{d^2}.
\end{align*}
Thus, by the union bound over all the at most $\delta d$ possible values of $k$, \textbf{whp} there are at most $\delta d\cdot nd^3\exp\left\{-\epsilon^3\log^2d\right\}=o_d(n)$ vertices in components of order between $\log^2d$ and $\delta d$.
\end{proof}

We can now bound the number of vertices in $|V_L(G(j))|$.

\begin{lemma}\label{l: |V_L(G(j))|}
For all $j\in \{1,2,3\}$, \textbf{whp} $|V_L(G(j))|\in \left[\left(1-o_d(1)\right)y(\epsilon-2\delta)n, \left(1+o_d(1)\right)y(\epsilon)n\right]$.
\end{lemma}
\begin{proof}
By Lemma \ref{l: |V_S(G(1))|}, we have that \textbf{whp} for all $j\in \{1,2,3\}$, $$|V_S(G(j))|\in \left[\left(1-o_d(1)\right)\left(1-y(\epsilon)\right)n,\left(1+o_d(1)\right)\left(1-y(\epsilon-2\delta)\right)n\right].$$ Furthermore, by Lemma \ref{l: gap statement}, \textbf{whp} for all $j\in \{1,2,3\}$, $|V(G)\setminus \left(V_S(G(j))\cup V_L(G(j))\right)|=o_d(n)$. Therefore, \textbf{whp} for all $j\in \{1,2,3\}$,
\begin{align*}
    &|V_L(G(j))|\le n+o_d(n)-\left(1-o_d(1)\right)\left(1-y(\epsilon)\right)n=\left(1+o_d(1)\right)y(\epsilon)n, \mathrm{and,}\\
    &|V_L(G(j))|\ge n-o_d(n)-\left(1-o_d(1)\right)\left(1-y(\epsilon-2\delta)\right)n=\left(1-o_d(1)\right)y(\epsilon-2\delta)n,
\end{align*}
as required.
\end{proof}

We now turn to the task of estimating $|W_L(G(2))|$. We begin with a lemma bounding from above the probability that a vertex has many neighbours which do not lie in large components in $G(1)$. Throughout the rest of the section, we suppose we have enumerated the vertices of $V(G)$ according to some arbitrary ordering.
\begin{lemma}\label{l: W_L(G(2)) aux}
Let $v\in V(G)$. The probability that there exists $U\subseteq N_G(v)$ such that
\begin{align*}
    \left|\bigcup_{u\in U}C_{G(1)}(u)\right|\in \left[\delta d, 2\delta d\right]
\end{align*}
and $|U|\le \delta^2d$ is at most $\exp\left\{-\delta^2 d\right\}$.
\end{lemma} 
\begin{proof}
We restrict ourselves to $U'\subseteq U$, such that $C_{G(1)}(u)$ are disjoint for each $u\in U'$. Let $F$ be a spanning forest of the components meeting $U'$ in $G(1)$ such that $|V(F)|=k\in\left[\delta d , 2\delta d\right]$. This forest is composed of some $1\le \ell \le \delta^2d$ tree components, $B_1, \ldots, B_{\ell}$, so we may assume that for every $i\neq j$, $B_i\cap B_j=\varnothing$ and all the edges leaving each $B_i$ are not in $G(1)$. Each $B_i$ contains a unique vertex $u_i \in U'$ for $1\le i \le \ell$. Note that if there is a subset $U \subseteq N_G(v)$ satisfying the conditions of the lemma, then such an $F$ exists.

Let us now bound from above the probability such a forest $F$ exists. There are at most $\sum_{m=1}^{\delta^2d}\binom{d}{m}\le \left(\frac{e}{\delta^2}\right)^{\delta^2d}$ ways to choose $U'\subseteq N_G(v)$ with $|U'|\le \delta^2d$. We can then specify the forest $F$ by choosing $|V(F)|=k \in [\delta d, 2\delta d]$, the number $1\leq \ell \leq \delta^2 d$ of tree components, their sizes $|B_i| = k_i$ such that $\sum_{i=1}^{\ell} k_i = k$ and finally the tree components $\{B_1,\ldots, B_\ell\}$, for which, by the second inequality in Lemma \ref{l: trees}, we have at most $\prod_{i=1}^\ell (ed)^{k_i-1} = (ed)^{k-\ell}$ choices (note that there is no factor of $n$ in the estimate since the roots of the tree components are determined --- these are the vertices in $U'$). For a fixed forest $F$, there are $k-\ell$ edges which must appear in $G(1)$, which happens with probability $\rho_1^{k-\ell}$. Since $|V(F)| \leq 2\delta d$ there are at least $k(d-2\delta d)$ edges in the boundary of $F$ which must not appear in $G(1)$, which happens with probability at most $(1-\rho_1)^{k(d-2\delta d)}$. Thus, by the union bound, the probability such $F$ exists is at most
\begin{align*}
     \left(\frac{e}{\delta^2}\right)^{\delta^2d}\sum_{k=\delta d}^{2\delta d}\sum_{\ell=1}^{\delta^2d}\sum_{\substack{k_1,\ldots,k_{\ell}\\k_1+\cdots+k_{\ell}=k}}(ed)^{k-\ell}\rho_1^{k-\ell}(1-\rho_1)^{k(d-2\delta)d}.
\end{align*}

Since $\rho_1\ge 1+\epsilon-2\delta$ and as $x\exp\{-ax\}$ is decreasing when $ax>1$, 
\begin{align*}
    (ed)^{k-\ell}\rho_1^{k-\ell}(1-\rho_1)^{k(d-2\delta)d}&\le \left[e(1+\epsilon-2\delta)\exp\left\{-(1-2\delta)(1+\epsilon-2\delta)\right\}\right]^k\\
    &\le \left[\exp\left\{\epsilon-2\delta-\frac{(\epsilon-2\delta)^2}{3}\right\}\exp\left\{-\epsilon+5\delta\right\}\right]^k\\
    &\le \exp\left\{-\epsilon^3k\right\},
\end{align*}
where the penultimate inequality follows from $1+x \le \exp\left\{x-\frac{x^2}{3}\right\}$ for small enough $x$, and the last inequality follows since $\delta$ is sufficiently small with respect to $\epsilon$. There are $\binom{k-1}{\ell-1}$ ways to choose $k_1,\ldots, k_{\ell}$ such that $\sum_{i=1}^{\ell}k_i=k$. Recalling that $k\in \left[\delta d, 2\delta d\right]$ and $\ell\in \left[\delta^2d\right]$, we have that $\binom{k-1}{\ell-1}\le \left(\frac{2e}{\delta}\right)^{\delta^2d}$. Altogether, the probability that such $F$ exists is at most
\begin{align*}
    \left(\frac{2e}{\delta^2}\right)^{2\delta^2d}\exp\left\{-\epsilon^3\delta d\right\}\le \exp\left\{-\frac{\epsilon^3 \delta d}{2}\right\}\le \exp\left\{-\delta^2d\right\},
\end{align*}
where we used the fact that $\delta$ is sufficiently small with respect to $\epsilon$.
\end{proof}

We are now ready to bound $|W_L(G(2))|$. Recall that $$W_L(G(2))=\left\{w\in V(G)\colon |N_G(w)\cap V_{L}(G(2))|\ge \delta^5d\right\}.$$ 

\begin{lemma}\label{l: |W_L(G(2))|}
\textbf{Whp}, $|W_L(G(2))|\ge \left(1-\exp\left\{-\delta^6d\right\}\right)n$.
\end{lemma}
\begin{proof}
Let $v\in V(G)$, and let $N_G(v)=\{u_1,\ldots, u_d\}$. Let us bound the probability that ${v\notin W_L(G(2))}$. To that end, we consider a modified BFS process, during which we will keep track of the sets $Q$ and $W$ as in Section \ref{s: BFS}, as well as the set $X_L$ of vertices in large components (of order at least $\delta d$) in $G(2)$, and a set $X_S$ of vertices in components which might not be large, where initially $X_L = X_S = \varnothing$.

The process runs in a number of \textit{epochs}, at the start of which the queue $Q$ is empty. We begin the next epoch by moving the first $u_i\in N_G(v)\setminus (X_L\cup X_S)$ into $Q$. We then run the BFS algorithm in $G(1)$ starting at $u_i$, where we copy every vertex discovered during the epoch to a set $K$, and we empty $K$ at the end of every epoch. We call $u_i$ the \textit{root} of the epoch. We then modify the BFS algorithm in the following ways: 
\begin{itemize}
    \item If $|K|=\delta d$ at some moment during the epoch, then once we finish exploring the connected component in $G(1)$, we move all the vertices in $K$ to $X_L$.
    \item Otherwise, if at some moment during the epoch, the first vertex $w$ in the queue $Q$ has at least $\delta d$ neighbours in $X_L$, we expose the edges in $G_{p_2}$ between $w$ and $X_L$. If any of the queries are successful, we end the epoch, move all the vertices in $Q$ to the set $W$ of vertices whose exploration is complete, and move all the vertices in $K$ to $X_L$. Otherwise, we continue exploring the component in $G(1)$.
    \item Otherwise, if we have not moved $u_i$ to $X_L$ due to the first two items, and we are at the end of the epoch since $Q=\varnothing$, we query all the unqueried edges between $K$ and $X_L$ in $G(1)$. If any of these queries are successful, we move all the vertices in $K$ to $X_L$. Otherwise, we move all the vertices in $K$ to $X_S$.
\end{itemize}
We stop the process either after $\delta^2d$ epochs, or once $N_G(v)\setminus (X_L\cup X_S)=\varnothing$.

Note that if all the vertices in $K$ moved to $X_S$ at the end of an epoch rooted at $u_i$, then the component of $u_i$ in $G(1)$ has size at most $\delta d$, since all edges incident to this component in $G(1)$, and potentially some in $G(2)$, were exposed during this epoch. We thus claim that by Lemma \ref{l: W_L(G(2)) aux} at the end of this process $|X_S|\le 2\delta d$ with probability at least $1-\exp\left\{-\delta^2d\right\}$. Indeed, let $Y$ be the set of all roots $u_i$ of epochs with $u_i\in X_S$. Since there are at most $\delta^2d$ epochs, we have that $|Y|\le \delta^2d$. Suppose that at the end of this process $|X_S|>2\delta d$. Consider the first moment when $|X_S|\ge \delta d$, and let $U\subseteq Y$ be the set of all roots $u_i$ of epochs with $u_i\in X_S$ at that moment. Note that since $|X_S|$ changes in size by at most $\delta d$ at each time, at that moment we have that $|X_S|\le 2\delta d$, and thus $|X_S|\in \left[\delta d, 2\delta d\right]$. Therefore, $U$ satisfies the conditions of Lemma \ref{l: W_L(G(2)) aux}.

Now, conditioned on $|X_S|\le 2\delta d$, at the end of the process we have two cases. If we had less than $\delta^2d$ epochs, then the process stopped since $N_G(v)\subseteq X_L\cup X_S$. However then, since $|X_S|\le 2\delta d$, we have that $|X_L|\ge (1-2\delta)d>\delta^5d$, and thus $v\in W_L(G(2))$. We may thus assume we had $\delta^2d$ epochs.

We claim that, if $|X_S| \leq 2 \delta d$ at the start of an epoch, then, conditioned on the full history of the process so far, the probability that the root $u_i$ of the epoch is moved into $X_L$ at the end of the epoch is at least $\frac{\delta^2}{2}$. Indeed, assume that we have yet to discover $\delta d$ vertices during the epoch rooted at $u_i$. If during this epoch no vertices with more than $\delta d$ neighbours in $X_L$ were discovered, then by assumption each vertex discovered during the epoch has at most $4\delta d$ neighbours which have already been discovered when it is queried --- at most $\delta d$ from the current epoch, at most $\delta d$ from $X_L$ and at most $2\delta d $ from $X_S$. Therefore, we can couple the BFS exploration in this epoch with a Galton-Watson tree $B$ with offspring distribution $Bin\left((1-4\delta)d,p_1\right)$, such that the BFS exploration stochastically dominates $B$ as long as $|B|\le \delta d$. Hence, standard results imply that with probability at least $y((1-4\delta)dp_1)\ge \delta$, $|B|$ grows to infinity, and in particular the BFS exploration discovers at least $\delta d$ vertices. 

Conversely, if during this epoch a vertex $w$ with at least $\delta d$ neighbours in $X_L$ was discovered, then, the probability that $u_i\in X_L$ is at least the probability that one of the $\delta d$ edges from $w$ to $X_L$ lies in $G_{p_2}$, which is at least 
\begin{align*}
    1-(1-p_2)^{\delta d}\ge 1-\exp\left\{-\delta dp_2\right\}= 1-\exp\left\{-\delta^2\right\}\ge \frac{\delta^2}{2},
\end{align*}
where the last inequality follows since $\exp\left\{-x\right\}\le 1-\frac{x}{2}$ for small enough $x$.

Hence, conditioned on the event that $|X_S| \leq 2\delta d$, which holds with probability at least $1-\exp\left\{-\delta^2d\right\}$ by Lemma \ref{l: W_L(G(2)) aux}, and assuming that we had $\delta^2d$ epochs, then $|X_L\cap N_G(v)|$ stochastically dominates $Bin\left(\delta^2d, \frac{\delta^2}{2} \right)$. Hence, by Lemma \ref{l:  chernoff}, with probability at least $1-\exp\left\{-\delta^5d\right\}$ we have that $|X_L\cap N_G(v)|\ge \delta^5d$. 

Therefore, putting all these together, the probability that $|X_L| < \delta^5 d$ is at most $$\exp\left\{-\delta^2d\right\} + \exp\left\{-\delta^5d\right\} \leq 2\exp\left\{-\delta^5d\right\}.$$ It follows that the probability that any particular $v\in V(G)$ is not in $W_L(G(2))$ is at most $2\exp\{-\delta^5d\}$, and so by Markov's inequality, \textbf{whp} $|V(G)\setminus W_L(G(2))|\le \exp\{-\delta^6d\}n$.
\end{proof}

We are now ready to prove Theorem \ref{th: weak expander}.
\begin{proof}[Proof of Theorem \ref{th: weak expander}]
By Lemma \ref{l: |V_L(G(j))|}, \textbf{whp} 
\begin{align*}
    |V_L(G_p)|\le\left(1+o_d(1)\right)y(\epsilon)n, \quad |V_L(G(2))|\ge \left(1-o_d(1)\right)y(\epsilon-2\delta)n\ge y(\epsilon)n-5\delta n.  
\end{align*}

We claim that, after sprinkling with probability $p_3$, \textbf{whp} all but of most $\delta n$ of the vertices in $V_L(G(2))$ merge into a single component. Let $A\sqcup B$ be a partition of the components forming $V_L(G(2))$, satisfying $\delta n \le k \coloneqq |A|\le |B|$. Let $A'$ be $A$ together with all the vertices in $V(G)\setminus B$ which have at least $\frac{\delta^5d}{2}$ neighbours in $A$, and let $B'$ be $B$ together with all the vertices in $V(G)\setminus A'$ which have at least $\frac{\delta^5d}{2}$ neighbours in $B$. Note that $A'\cap B'=\varnothing$ and $V_L(G(2))\cup W_L(G(2))\subseteq A'\cup B'$.

Fix the partition $A\sqcup B$ of (the components forming) $V_L(G(2))$. Note that $|A'|, |B'|\ge k$ and one of them is of order at most $\frac{n}{2}$. We may assume without loss of generality that $k \le |A'|\le \frac{n}{2}$. Thus, by our assumption on the isoperimetric constant of $G$, that is, $i(G)\ge C$, there are at least $Ck$ edges in $G$ between $A'$ and $(A')^C$. Since $G$ is $d$-regular, and by Lemma \ref{l: |W_L(G(2))|} \textbf{whp} $|V(G)\setminus W_L(G(2))|\le \exp\left\{-\delta^6d\right\}n=o(n/d)$, it follows that $e_G(A', B')\ge Ck/2$.

Let us now expose these edges (between $A'$ and $B'$) with probability $p_3$. By Lemma \ref{l: matching}, with probability at least $1-\exp\left\{-\frac{\delta^2Ck}{d}\right\}$ there is a matching of size at least $\frac{\delta^2Ck}{d}$ between $A'$ and $B'$ in $G_{p_3}$.
\begin{figure}[H]
\centering
\includegraphics[width=0.8\textwidth]{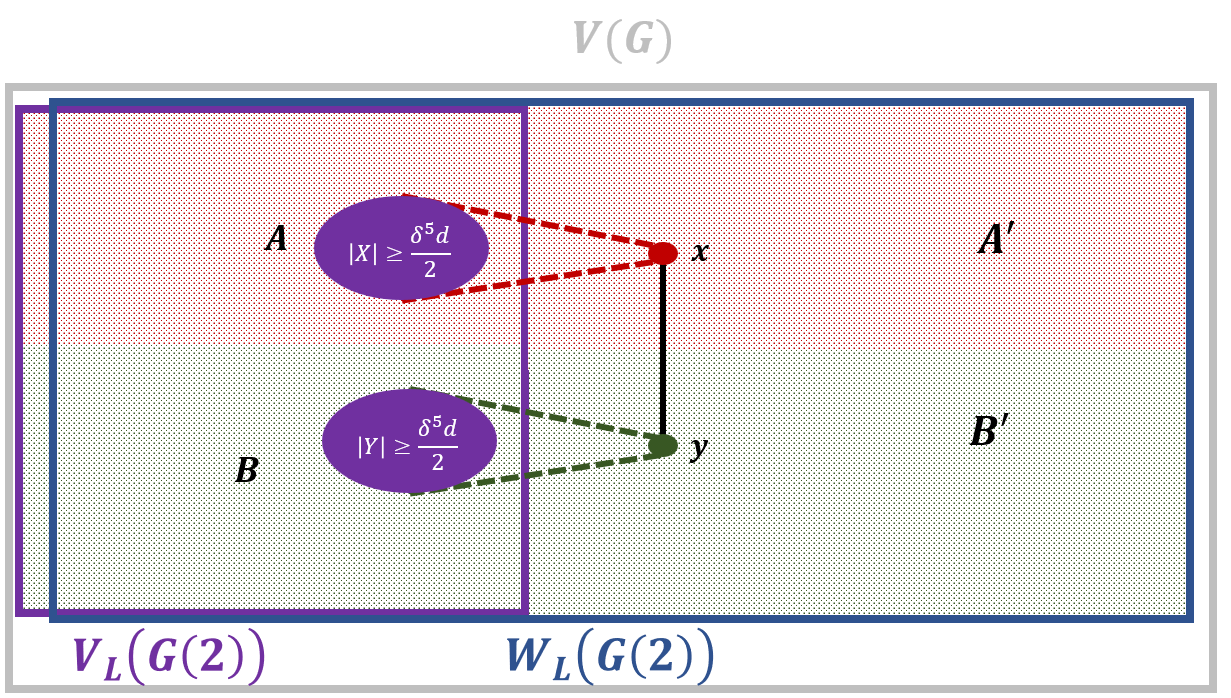}
\caption{Illustration of the sets utilised in the proof of Theorem \ref{th: weak expander} and their properties. By Lemma \ref{l: |W_L(G(2))|}, \textbf{whp} $|W_L(G(2))|=n-o(n/d)$ (appears in dark blue). $V_L(G(2))$ appears in purple, together with a partition of it into $A$ (red-dotted) and $B$ (green-dotted). This extends to $A'$ (red-dotted) and $B'$ (green-dotted), which cover $W_L(G(2))$. In the figure, the edge $xy$ between $A'$ and $B'$ appears in solid black. $X$, the neighbourhood of $x$ in $A$, appears in purple, and its size is lower-bounded by construction. Similarly, $Y$, the neighbourhood of $y$ in $B$, appears in purple with its size also lower-bounded by construction. When sprinkling with probability $p_3$, with very high probability there will be many black edges between $A'$ and $B'$ in $G_{p_3}$, which we can then extend to paths of length three between $A$ and $B$.}
\label{f: proof}
\end{figure}

Assume such a matching exists. By the definition of $A'$ and $B'$, from the endpoints of each edge in the matching there are at least $\frac{\delta^5d}{2}$ edges into $A$, and at least $\frac{\delta^5d}{2}$ edges into $B$. The probability that a given edge in the matching does not extend to a path of length three between $A$ and $B$ is thus at most $2(1-p_3)^{\delta^5d/2}\le 2\exp\left\{-\delta^6/2\right\}\le \exp\left\{-\delta^7\right\}$. These events are independent for the edges of the matching, and thus the probability that there is no path in $G_{p_3}$ connecting $A$ and $B$ is at most $$\exp\left\{-\frac{\delta^{7}\cdot \delta^2Ck}{d}\right\}\le \exp\left\{-\frac{C\delta^{10}n}{d}\right\}.$$

Since there are at most $2^{\frac{n}{\delta d}}$ possible component-respecting partitions $A\sqcup B$ of $V_L(G(2))$, by choosing $C= \frac{2}{\delta^{20}}$, we conclude that with probability at least
\begin{align*}
    1-\exp\left\{-\frac{C\delta^{10}n}{d}+\frac{n}{\delta d}\right\}\ge 1-\exp\left\{-\frac{\sqrt{C}n}{d}\right\},
\end{align*}$G_p$ contains a component of order at least $y(\epsilon)n-5\delta n-\delta n=y(\epsilon)n-\frac{10}{C^{1/20}}n$, as required. Furthermore, the second-largest component is of order at most $(1+o_d(1))\frac{10n}{C^{1/20}}$.
\end{proof}

\begin{remark}\label{r: key theorem}
A careful reading of the proof of Theorem \ref{th: weak expander} shows that we do not use the full extent of our assumption $i(G)\ge C$. In fact, a milder requirement suffices: for any $S\subseteq V(G)$ with $\delta n \le |S|\le \frac{n}{2}$, we require that $\partial_G(S)\ge C\delta n$. Furthermore, we note that it suffices to choose $\delta=\frac{\epsilon^2}{10}$, and so we can take $C_0=\Omega\left( \delta^{-20}\right)=\Omega\left(\epsilon^{-40}\right)$. Finally, we note that key techniques of the proof of Theorem \ref{th: weak expander} can easily be adapted to prove that a percolated hypercube $Q^d_p$ in the supercritical regime, when $p=\frac{1+\epsilon}{d}$, typically contains a unique giant component of the correct asymptotic order, namely, $y(\epsilon)|V(Q^d)|$. Indeed, by Harper's inequality \cite{H64}, which states that $i_k(Q^d)\ge d-\log_2 k$, we have that the edge-boundary of an arbitrary set $S$ of order $|S|=\frac{|V(Q^d)|}{2^C}=2^{d-C}$ has size at least $C|S|$. Following the lines of our proof arguments of Theorem \ref{th: weak expander}, we can show that \textbf{whp} $Q^d_p$ contains constantly many components of linear order, whose sizes sum up asymptotically to $y(\epsilon)|V(Q^d)|$. Using a standard sprinkling argument, we can then show that \textbf{whp} these linear sized components merge into a single giant component.
\end{remark}

\section{Tightness of Theorem \ref{th: weak expander}}\label{s: construction}
We begin by describing the construction which will be used for the proof of Theorem \ref{th: construction}.

\subsection{Construction for Theorem \ref{th: construction}}
\paragraph{First construction}\label{first construction}
Let $d_1, d$, and $n$ be even integers satisfying the following conditions:
\begin{align}\label{e:dd_1}
d = \omega(1)  \text{ is such that } \, d+2 \text{ divides } n,\,\, d_1=\frac{d}{Ce^2}, \text{ and } n=\omega\left(d^2\right).
\end{align}
Let $H_0$ be an $\left(\frac{n}{d+2},d_1,\lambda\right)$-graph with $\lambda=O\left(\sqrt{d_1}\right)$ (note that $d_1\cdot\frac{n}{d+2}$ is even, since $d_1$ is even and $d+2$ divides $n$).
Indeed, \textbf{whp} a random $d_1$-regular graph on $\frac{n}{d+2}$ vertices satisfies this (see \cite{KS06}, \cite{TP19}, \cite{S23} and the references therein for even stronger results). Note that, by the expander mixing lemma (see \cite{AC88}), it follows that $i(H_0)\ge \frac{d_1}{3}$.

For every $v\in V(H_0)$, let $F_0(v)$ be a $(d+1)$-clique and let $M(v) \subseteq E\left(F_0(v)\right)$ be an arbitrary matching of size $\frac{d-d_1}{2}$. Let $F_1(v) = F_0(v) - M(v)$ be the graph obtained by deleting the edges of this matching from the clique.

We now consider the graph $G$ formed by taking the union of $H_0$ and $\bigcup_{v \in V(H_0)} F_1(v)$ and joining each $v \in V(H_0)$ to each vertex in $V(M(v)) \subseteq V(F_1(v))$ (see Figure \ref{f: const 1} for an illustration of $G$).
\begin{figure}[H]
\centering
\includegraphics[width=0.5\textwidth]{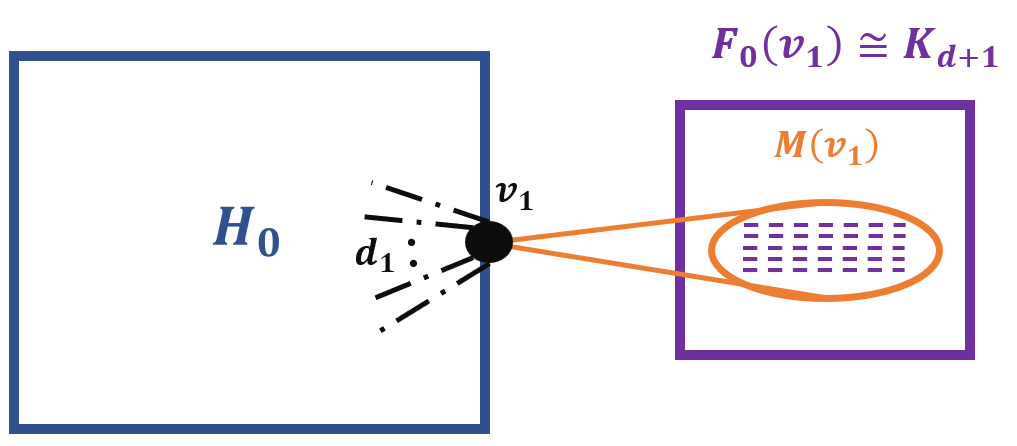}
\caption{An illustration of the \nameref{first construction}}
\label{f: const 1}
\end{figure}

Let us first show that our graph $G$ satisfies the assumptions of Theorem \ref{th: construction}, noting that $\frac{d_1}{5d}\ge \frac{1}{5Ce^2}\ge \frac{1}{40C}$.
\begin{claim}\label{properties of G}
Let $G$ be the graph built according to the \nameref{first construction}. Then $G$ is a $d$-regular $n$-vertex graph with $i(G)\ge \frac{d_1}{5d}$.
\end{claim}
\begin{proof}
We first note that $G$ is $d$-regular. Indeed, every vertex $v\in V(H_0)$ has $d_1$ neighbours in $H_0$, and $d-d_1$ neighbours in $V\left(M(v)\right)\subseteq V\left(F_1(v)\right)$. Furthermore, every vertex in $V\left(F_1(v)\right)\setminus V\left(M(v)\right)$ has $d$ neighbours in $G$, since $F_0(v)\cong K_{d+1}$ and the degree of vertices not in $V\left(M(v)\right)$ has not been changed. 

Furthermore, $G$ has $n$ vertices. Indeed, $\left|V(H_0)\right| = \frac{n}{d+2}$ and for each $v \in V(H_0)$, $|F_1(v)| = d+1$. Therefore,
\begin{align*}
    |V(G)|=\frac{n}{d+2}+(d+1)\frac{n}{d+2}=n.
\end{align*}

Let us now show that $i(G)\ge \frac{d_1}{5d}$. Given $S\subseteq V(G)$ with $|S|\le \frac{n}{2}$, let $S_1=S\cap V(H_0)$ and for every $v\in V(H_0)$, let $S_F(v)=S\cap \left(\{v\}\cup V\left(F_1(v)\right)\right)$.

Suppose first that $|S_1|\ge \frac{4|S|}{5(d+1)}$ and $|S_1|\le \frac{\left|V(H_0)\right|}{2}$. Then, since
\[
\frac{e(S,S^C)}{|S|}\ge \frac{e_{H_0}(S_1, S_1^C)}{|S|},
\]
it follows from the fact that $i(H_0) \geq \frac{d_1}{3}$ that
\begin{align*}
    \frac{e_{H_0}(S_1, S_1^C)}{|S|}\ge \frac{\frac{d_1|S_1|}{3}}{|S|}\ge \frac{d_1}{4d}.
\end{align*}
If $|S_1|\ge\frac{4|S|}{5(d+1)}$ and $\frac{\left|V(H_0)\right|}{2}\le |S_1|\le \frac{2\left|V(H_0)\right|}{3}$, then
\begin{align*}
    \frac{e_{H_0}(S_1, S_1^C)}{|S|}\ge\frac{\frac{d_1|S_1^c|}{3}}{|S|} \geq \frac{d_1\left|V(H_0)\right|}{9|S|}\ge\frac{2d_1\left|V(H_0)\right|}{9(d+2)\left|V(H_0)\right|}\ge \frac{d_1}{5d}.
\end{align*}

Now, suppose that $|S_1|<\frac{4|S|}{5(d+1)}$. Then, we may assume that there is some subset $S_2 \subseteq S$ of size at least $\frac{|S|}{5}$ which lies in some union of $S_F(v)$ where $v \not\in S$. Then, since $S = \bigsqcup_{v \in V(H_0)} S_F(v)$ we can see that
\begin{equation}\label{e:split}
\frac{e(S,S^C)}{|S|}\ge \frac{1}{5} \min_{v \in S_2} \frac{e(S_F(v),S^C)}{|S_F(v)|}.
\end{equation}
However, for each $v \in S_2$ we have that $S_F(v)$ meets $F_1(v)$ and does not contain $v$. Since $F_1(v)$ is a $(d+1)$-clique with a matching removed, we have that
\begin{align*}
    e(S_F(v), S^C)&\ge e(S_F(v), F_1(v)\setminus S_F(v))\\
    &\ge |S_F(v)|\left(d+1-|S_F(v)|\right)-|S_F(v)|=|S_F(v)|\left(d-|S_F(v)|\right).
\end{align*}
Note that as long as $|S_F(v)|<d$, the above is at least $|S_F(v)|$. On the other hand, if $|S_F(v)|=d$, then since $S_F(v)$ does not contain $v$ and $v$ has $d-d_1$ neighbours in $F_1(v)\setminus \{v\}$, we have that
\begin{align*}
    e(S_F(v), S^C)&\ge e(v, S_F(v))\ge d-d_1-1.
\end{align*}
Hence, in this case by \eqref{e:split},
\[
\frac{e(S,S^C)}{|S|} \geq \frac{1}{5}\cdot\frac{d-d_1-1}{d} \geq \frac{d_1}{5d},
\]
where the last inequality holds by our assumption on $d_1$.

Finally, suppose that $|S_1|\ge \frac{2\left|V(H_0)\right|}{3}$. Then, since $|S|\le \frac{n}{2}$, there is a set $S_1'\subseteq S_1$, $|S_1'|\ge \frac{\left|V(H_0)\right|}{6}$, such that for every $v\in S_1'$, $S_F(v)\cap F_1(v)\neq F_1(v)$. In particular, each such copy contributes at least $d-d_1$ edges to the boundary of $S$. Therefore,
\begin{align*}
    \frac{e(S,S^C)}{|S|}\ge \frac{(d-d_1)\left|V(H_0)\right|}{6|S|}\ge \frac{d-d_1}{3(d+2)}\ge \frac{d_1}{5d},
\end{align*}
where once again the last inequality holds by our assumption on $d_1$.
\color{black}
\end{proof}
Note that $i(G)=\Theta\left(\frac{d_1}{d}\right)$. Indeed, let $S=\{v\}\cup F_1(v)$. Then $|S|=d+2$ and $\partial_G(S)=d_1$, and thus $i(G)\le \frac{d_1}{d+2}$.

Theorem \ref{th: construction} will be an immediate corollary of the following Lemma.
\begin{lemma}\label{l: technical construction}
Let $G$ be the graph built according to the \nameref{first construction}. Let $p=\frac{C}{d}$. Then, \textbf{whp}, all the components of $G_p$ have size at most $3d\log n$.
\end{lemma}

\begin{proof}[Proof of Lemma \ref{l: technical construction}]
Let $k\ge 3d\log n$ and let $\mathcal{A}_k$ be the event that $G_p$ contains a component of size $k$. Note that if $\mathcal{A}_k$ occurs, then there is a tree $T$ of size $k$ in $G$, of all whose edges lie in $G_p$. It follows that there is a subtree $T'\subseteq T$ of size at least $k_0\coloneqq \frac{k}{d+1}\ge 2\log n$, such that $T'\subseteq H_0\subseteq G$, all of whose edges lie in $G_p$. Recalling that $H_0$ is a $d_1$-regular graph on $\frac{n}{d+2}$ vertices, by Lemma \ref{l: trees} there are at most $\frac{n}{d+2}(ed_1)^{k_0-1}$ ways to choose $T'$, and we retain its edges in $G_p$ with probability $p^{k_0-1}$. Therefore, since $pd_1 = \frac{1}{e^2}$, we have
\begin{align*}
    \mathbb{P}\left[\mathcal{A}_k\right]\le \frac{n}{d+2}\cdot (ed_1)^{k_0-1}p^{k_0-1}=\frac{n}{d+2}\cdot \left(\frac{1}{e}\right)^{k_0-1}=o\left(\frac{1}{n}\right).
\end{align*}
Hence, by the union bound, \textbf{whp} $\mathcal{A}_k$ does not hold for all $3d \log n \leq k \leq n$.
\end{proof}

\subsection{Construction for Theorem \ref{th: construction 2}}
\paragraph{Second construction}\label{second construction} 
Let $C, d, n$, and $p$ be as in the statement of Theorem \ref{th: construction 2}. Let $C_1\coloneqq 3C$, let $d_1\coloneqq d-C_1$ and let $t\coloneqq \frac{n}{d_1+1}$. We further assume that $C_1, d_1$, $t$, and $n$ satisfy the needed parity assumptions for what follows.

Let $H$ be a $C_1$-regular graph on $n$ vertices, with $i(H)\ge \frac{C_1}{3}=C$ (indeed, \textbf{whp} a random $C_1$-regular graph on $n$ vertices satisfies this). Since $t=\frac{n}{d_1+1}=\frac{n}{d-3C}\ge \frac{n}{d}\ge 10C$, we conclude that there exists an equitable (proper) colouring of $H$ in $t$ colour classes, $A_1,\ldots, A_t$, with each colour class containing exactly $d_1+1$ vertices \cite{HS70}. We form $G$ by adding to $H$ all the possible edges in $H[A_j]$, that is, $G[A_j]\cong K_{d_1+1}$, for every $j\in[t]$. Since $d=d_1+C_1$, note that $G$ is a $d$-regular graph on $t(d-3C+1)$ vertices. Furthermore, $i(G)\ge i(H)\ge C$. 

With this construction at hand, we are ready to prove Theorem \ref{th: construction 2}.
\begin{proof}[Proof of Theorem \ref{th: construction 2}]
Note that, for every $j\in [t]$, the edges between $A_j$ and $V(G)\setminus A_j$ are those in $H$. Let $X$ be the number of sets $A\in\{A_1, \ldots, A_t\}$, such that $e_{H_p}(A, V(H)\setminus A)=0$.

For each fixed $j\in[t]$ we have that $e_H(A_j, V(H)\setminus A_j)=C_1(d_1+1)$. The probability this set is disjoint from the rest of the graph in $H_p$ is 
\begin{align*}
    (1-p)^{C_1(d_1+1)}\ge \exp\left\{-2\frac{C_1(d-C_1+1)}{d}\right\}\ge \exp\left\{-4C_1\right\},    
\end{align*}
where we used $d\ge 7C=\frac{7C_1}{3}$. Hence, $\mathbb{E}[X]\ge \exp\{-4C_1\}t$. Now, note that changing one edge can change the value of $X$ by at most two. Hence, by Lemma \ref{l: azuma},
\begin{align*}
    \mathbb{P}\left[X\le \frac{\mathbb{E}[X]}{2}\right]\le 2\exp\left\{-\frac{\exp\{-8C_1\}t^2/4}{2\cdot C_1np\cdot 4}\right\}\le 2\exp\left\{-\frac{\exp\{-8C_1\}t}{12}\right\}.
\end{align*}
Thus, with probability at least $1-\exp\left\{-\exp\{-10C_1\}t\right\}\ge 1-\exp\left\{-\frac{n}{\exp\{30C\}d}\right\}$, there are at least $\frac{\exp\{-4C_1\}t}{2}$ disjoint sets $A\in \{A_1, \ldots, A_t\}$ in $G_p$. Since $G[A_j]\cong K_{d_1+1}$ for every $j\in [t]$, we have that with probability $1-o_d(1)$ there exists a component of order $\left(2\epsilon-O(\epsilon^2)\right) d_1$ in $G_p[A_j]$. Therefore, with probability at least $1-\exp\left\{-\frac{n}{\exp\{30C\}d}\right\}-o_d(1)$, there are at least two components of order $\left(2\epsilon-O(\epsilon^2)\right) d_1\ge \epsilon d$ in $G_p$, where we used our assumption that $d\ge 7C$.
\end{proof}

\section{Discussion and avenues for future research}\label{s: discussion}
In this paper, we aimed to capture the minimal requirements on a host graph $G$, guaranteeing the percolated subgraph $G_p$ undergoes a similar phase transition, with respect to the size of the largest component, as that in $G(n,p)$. Theorem \ref{th: weak expander}, together with the two constructions (Theorems \ref{th: construction} and \ref{th: construction 2}) showing that it is qualitatively tight, gives a qualitative answer to this question. In Theorems \ref{th: general statement b} and \ref{th: general statement}, we further demonstrated the intrinsic connection between the isoperimetry of the host graph, both global and local, to the existence of large components in the percolated subgraph.

There are several natural questions which arise in this context. While Theorem \ref{th: weak expander} gives a qualitative answer to Question \ref{q: refined}, one can ambitiously aim for a more precise \textit{quantitative} answer as to the minimum requirements on $i(G)$, or a more restricted notion of expansion, which guarantees this behaviour.
\begin{question}\label{q: C and eps}
Given $\epsilon>0$ and $d\to\infty$, what is the minimal $C(\epsilon,d)$ such that every $d$-regular $n$-vertex graph $G$ with $i(G) \geq C(\epsilon,d)$ is such that $G_p$ \textbf{whp} contains a unique component of linear order, specifically one of order asymptotic $y(\epsilon)n$, when $p=\frac{1+\epsilon}{d}$?
\end{question}
Theorem \ref{th: weak expander} implies that any function $C(\epsilon,d)=\omega(1)$ is sufficient, and Theorem \ref{th: construction} implies that it is necessary to take $C(\epsilon,d)\geq \frac{1}{40(1+\epsilon)}$. It would be interesting to determine if the answer is in fact independent of $d$ (or even $\epsilon$). Furthermore, we note that the regularity assumption in Theorem \ref{th: weak expander} can be somewhat weakened. With slight technical care, we can assume instead that $\frac{\Delta(G)}{\delta(G)}=1+o_d(1)$, where $\Delta(G)$ is the maximum degree of $G$ and $\delta(G)$ is the minimum degree of $G$. It could be interesting to determine the (irregular) degree assumptions that are necessary for Theorem \ref{th: weak expander} to hold.

It is known (see, for example, \cite[Theorem 1]{DEKK22}) that in subcritical percolation on any $d$-regular graph the largest component typically has logarithmic order. Theorem \ref{ER thm} demonstrates that in $G(n,p)$ the second-largest component is of logarithmic order in the supercritical regime, which is a particularly simple facet of the \emph{duality principle} which holds for $G(n,p)$, which broadly says if we remove the giant component from a supercritical random graph, what remains resembles a subcritical random graph. Similar behaviour, at least in terms of the order of the second-largest component, is known to hold in many other percolation models such as percolation on $(n,d,\lambda)$-graphs \cite{FKM04} or on hypercubes \cite{BKL92} (and in fact on any regular high-dimensional product graph \cite{DEKK22}).

Whilst Theorem \ref{th: weak expander} shows that relatively weak assumptions on the expansion of $G$ are sufficient to demonstrate a threshold for the existence of a linear sized component, Theorem \ref{th: construction 2} shows that such assumptions alone are not sufficient to effectively bound the size of the second-largest component in the supercritical regime.

It is an interesting open problem to determine natural conditions on the host graph which guarantee that the phase transition quantitatively resembles that of $G(n,p)$.
\begin{question}\label{q: ERCP}
    Let $G$ be a $d$-regular $n$-vertex graph with $d\to\infty$. Let $\epsilon>0$ be a sufficiently small constant, and let $p=\frac{1+\epsilon}{d}$. What are the minimal requirements on $G$ such that \textbf{whp} the largest component in $G_p$ is of asymptotic order $y(\epsilon)n$, and the second-largest component is of order $O(\log n)$?
\end{question}

It is worth noting that whilst the assumptions of Theorem \ref{th: general statement} impose a restriction on the \textit{minimum} degree of the host graph $G$, an averaging argument shows that the assumptions of Theorem \ref{th: general statement b} impose a restriction only on the \emph{average} degree of $G$. Motivated by their work on percolation on graphs with large minimum degree, Krivelevich and Samotij \cite{KS14} ask whether there is a quantitatively similar threshold for the existence of a large component in graphs with large average degree. In particular, if $G$ has average degree $d$ and $p=\frac{1+\epsilon}{d}$, they ask if \textbf{whp} $G_p$ contains a component (or indeed, even a cycle) with $\Omega(d)$ vertices. 

As we mentioned before, our argument of Theorem \ref{th: weak expander} allows also for $\epsilon\coloneqq \epsilon(d)$ which tends \textit{slowly} to zero as $d$ tends to infinity. We have not tried to optimise this dependency. It could be very interesting to analyse the correct dependency, and in particular to see what could be said about the size of the largest component, and of the component sizes in general, in the slightly supercritical regime.
\bibliographystyle{abbrv}
\bibliography{perc} 
\end{document}